\documentclass[reqno,centertags,12pt,draft]{amsart}
\usepackage[letterpaper,margin=1.3in]{geometry}
\usepackage{amsmath,amsthm,amsfonts,amssymb,enumerate}
\usepackage{bigints}

\newcommand{\bb}{\mathbb}

\newcommand{\ca}{\mathrm{Cap}}

\newtheorem{theorem}{Theorem}

\newtheorem{corollary}[theorem]{Corollary}
\theoremstyle{definition}

\theoremstyle{remark}

\numberwithin{equation}{section}
\numberwithin{theorem}{section}

\begin{document}
\title{ Extremal polynomials on a Jordan arc}

\author{G\"{o}kalp Alpan}
\address{Department of Mathematics, Uppsala University, Uppsala, Sweden}
\email{gokalp.alpan@math.uu.se}


\begin{abstract}
Let $\Gamma$ be a $C^{2+}$ Jordan arc and let $\Gamma_0$ be the open arc which consists of interior points of $\Gamma$. We find concrete upper and lower bounds for the limit of Widom factors for $L_2(\mu)$ extremal polynomials on $\Gamma$ which was given in \cite{Wid69}. In addition, we show that the upper bound for the limit supremum of Widom factors for the weighted Chebyshev polynomials which was obtained in \cite{Wid69} can be improved once two normal derivatives of the Green function do not agree at one point $z\in \Gamma_0$. We also show that if $\Gamma_0$ is not analytic then we have improved upper bounds.
\end{abstract}

\date{\today}
\maketitle
\nocite{*}

\section{Introduction}
We say that a Jordan arc $\Gamma$ is $C^{2+}$ smooth if it has a twice continuously differentiable parametrization $\gamma(t)$, $t\in[-1,1]$ such that $\gamma^{\prime}(t)\neq 0$ for each $t\in[-1,1]$  and $\gamma^{\prime\prime}$ is $\mathrm{Lip}c$ for some $c>0$. An open Jordan arc $\Gamma_0$ is called analytic if it has a parametrization $\gamma(t)$, $-1<t<1$, such that $\gamma$ can be expanded into a power series around each $t\in (-1,1)$ and $\gamma^{\prime}(t)\neq 0$ for $-1<t<1$. Throughout $\Gamma$ will denote a $C^{2+}$ Jordan arc and $\Gamma_0$ will denote the open Jordan arc which consists of all points in $\Gamma$ except the endpoints $A$ and $B$.

Let $\mu_\Gamma$ denote the equilibrium measure of $\Gamma$ and $\ca(\Gamma)$ denote the logarithmic capacity of $\Gamma$. The Green function $g_\Gamma$ for the region $\overline{\bb C}\setminus \Gamma$ is given by

\begin{align}
	g_\Gamma(z) = -\log\ca(\Gamma)+\int\log|z-\zeta|\,d\mu_\Gamma(\zeta), \quad z\in\bb C.
\end{align}
A finite (positive) Borel measure $\mu$ of the form $d\mu = f d\mu_\Gamma$ is in the Szeg\H{o} class $\mathrm{Sz}({\Gamma})$ if 
\begin{align}
S(\mu):= S(f)= \exp\left[\int \log{f} d\mu_\Gamma\right]>0.
\end{align}
Let $\rho$ be an upper semicontinuous non-negative  function which is bounded above on $\Gamma$ with $S(\rho)>0$. Let
\begin{align}
	t_n(\Gamma,\rho):= \inf_{P\in \Pi_n}\sup_{z\in\Gamma}|P(z)\rho(z)| 
\end{align}
where infimum is taken over all monic polynomials $\Pi_n$ of degree $n$. Then $n$-th Widom factor for the sup-norm with respect to the weight function $\rho$ is defined as 

\begin{align}
	W_{\infty,n}(\Gamma,\rho):=\frac{t_n(\Gamma,\rho)}{\ca(\Gamma)^n}.
\end{align}
For the case $\rho\equiv 1$, we use the notation $W_{\infty,n}(\Gamma,1)$. It is trivial in this case that $S(\rho)=1$.

For a measure $\mu\in \mathrm{Sz}(\Gamma)$, let $P_n(\cdot; \mu)$ denote $n$-th monic orthogonal polynomial for $\mu$. Then $n$-th Widom factor for $\mu$ is defined as
\begin{align}
	W_{2,n}(\mu):=\frac{\|P_n(\cdot;\mu)\|_{L_2{(\mu)}}}{\ca(\Gamma)^n}.
\end{align}
For a deeper discussion of Widom factors we refer the reader to \cite{alper, AlpZin2, AndNaz18, CSZ3, CSZ4, Eic17, GonHat15, kali1, kali, SchZin, Tot09, Tot14, TotYud15, Wid69}. Before stating our results on Widom factors we have to introduce several concepts and summarize some key results from \cite{Wid69}. 

Let $\phi$ be the conformal map from $\overline{\bb C}\setminus \Gamma$  onto $\overline{\bb C}\setminus \bb D$ such that $\phi(\infty)=\infty$ and 
\begin{align}
\phi^{\prime}(\infty):= \lim_{z\rightarrow\infty} \frac{\phi(z)}{z}=\ca(\Gamma)^{-1}.
\end{align}

We call two sides of $\Gamma$ positive and negative sides. The map $\phi$ has continuously differentiable boundary values $\phi_{+}$ and $\phi_{-}$ except for the endpoints. Let 
\begin{align}
\omega_\Gamma(z):=\frac{1}{2\pi}(|\phi^{\prime}_{+}(z)|+|\phi^{\prime}_{-}(z)|).
\end{align}
Then
\begin{align}
	d\mu_\Gamma=\omega_\Gamma\, ds
\end{align}
where $ds$ is the arc-length measure on $\Gamma$.
If $n_{\pm}$ are the two normals at $z\in \Gamma$ then
\begin{align}\label{hadda}
g_{\pm}^{\prime}(z):=\frac{\partial g_\Gamma(z)}{\partial n_{\pm}}=|\phi^{\prime}_{\pm}(z)|.
\end{align}
Note that $g_{\pm}^{\prime}(z)>0$ for all $z\in \Gamma_0$ and both $g_{\pm}^{\prime}(z)\sqrt{|z-A||z-B|}$ can be extended to positive $C^{0+}$ functions on $\Gamma$, see \cite[Appendix]{Tot14a}. This in particular implies that $\omega_\Gamma(z)\sqrt{|z-A||z-B|}$ is positive and $C^{0+}$ on $\Gamma$.

 Let $\mu\in\mathrm{Sz}(\Gamma)$ with $d\mu=fd\mu_\Gamma$. Denote the harmonic function on $\overline{\bb C}\setminus \Gamma$ with boundary values $\log{(f \omega_\Gamma)}$ on $\Gamma$ by $h$. Its harmonic conjugate $\tilde{h}$ is chosen so that
\begin{align}
R_{\mu}(z):= \exp[h(z)+i\tilde{h}(z)]
\end{align}
is the analytic function on   $\overline{\bb C}\setminus \Gamma$ with (non-tangential) boundary values $|R_{\mu}(z)|=f(z) \,\omega_\Gamma(z)$ ($ds$ a.e.)  and $R_\mu(\infty)>0$. Here $R_\mu(z)\neq 0$ for $z\in \overline{\bb C}\setminus \Gamma$. Similarly, we define $R_f$ to be the analytic function  without any zeros  on $\overline{\bb C}\setminus \Gamma$ such that $|R(z)|=f(z)$ ($ds$ a.e.) on $\Gamma$ and $R_f(\infty)>0$. Then 

\begin{align}\label{R1}
	R_{\mu}(z)=R_{\mu_\Gamma}(z)R_f(z)
\end{align}
and
\begin{align}\label{S1}
	R_f(\infty)=S(f).
\end{align}
The above reformulation of $R_\mu$ is quite useful as it allows us to state and prove our results in terms of $\mu_\Gamma$.

We define $H_2(\overline{\bb C}\setminus \bb D)$ as the space of analytic functions $F$ on $\overline{\bb C}\setminus \bb D$ such that $G(z)=F(1/z)$ is in the usual Hardy space $H_2(\bb D)$.

Let $\Psi$ be the inverse of $\phi$. 
By $H_2( \overline{\bb C}\setminus \Gamma,\mu)$ we mean (see e.g \cite{kali1, kali}) analytic functions $F$ on  $\overline{\bb C}\setminus \Gamma$ such that 
\begin{align}
	F(\Psi)\sqrt{\frac{R_\mu(\Psi)}{\phi^{\prime}(\Psi)}}\in H_2(\overline{\bb C}\setminus \bb D).
\end{align}

These functions have non-tangential boundary values. 

The  norm is given by
\begin{align}\label{normi}
	\|F\|_{H_2( \overline{\bb C}\setminus \Gamma,\,\mu)}:=\left[\int (|F_{+}(z)|^2 +|F_{-}(z)|^2 ) \, d\mu(z)\right]^{1/2}
\end{align}

Define the quantity
\begin{align}\label{pro1}
\nu(\mu):= \inf _F	\|F\|_{H_2( \overline{\bb C}\setminus \Gamma,\,\mu)}^2
\end{align}
where the infimum is taken over all functions $F$ such that $F\in H_2( \overline{\bb C}\setminus \Gamma,\mu)$ with $F(\infty)=1$. Then
\begin{align}\label{expi}
	\nu(\mu)= 2\pi \frac{R_\mu(\infty)}{\phi^{\prime}(\infty)}=2\pi R_\mu(\infty) \ca{(\Gamma)}
\end{align}
and 
\begin{align}\label{extra}
F_\mu(z):= \sqrt{\frac{\phi^{\prime}(z)}{R_\mu(z)}}\sqrt{\frac{R_\mu(\infty)}{\phi^{\prime}(\infty)}}
\end{align}
is the unique extremal function for the extremal problem \eqref{pro1}, see \cite[Theorem 6.2]{Wid69}.

It follows from \cite[Theorem 12.3, Theorem 9.2]{Wid69}, \eqref{R1}, \eqref{S1} that

\begin{align}
\lim_{n\rightarrow\infty} [W_{2,n}(\mu)]^2&=2\pi R_\mu(\infty)\ca(\Gamma)\\
&=[2\pi R_{\mu_\Gamma}(\infty)\ca(\Gamma)]R_f(\infty)\\
&= \nu(\mu_\Gamma)S(f).\label{genel}
\end{align}
The next theorem is our first result and this estimation of $\nu(\mu_\Gamma)$ is used in the subsequent results.
\begin{theorem}\label{mainn}
	The quantity $\nu(\mu_\Gamma)$ satisfies 
	\begin{align}\label{nu1}
		1<\nu(\mu_\Gamma)\leq 2.
	\end{align}
The equality
\begin{align}\label{nu2}
	\nu(\mu_\Gamma)=2
\end{align}
holds if and only if
\begin{align}\label{sym}
g_{+}^{\prime}(z)=g_{-}^{\prime}(z)\,\, \mbox{ for all }  z\in \Gamma_0.
\end{align}
\end{theorem}
It was shown in \cite[Theorem 1]{Tot14} that $\liminf_{n\rightarrow\infty}W_{\infty,n}(\Gamma,1)>1.$ If $\mu\in \mathrm{Sz}(\Gamma)$ and $\mu(\Gamma)=1$, it is easy to see that $W_{2,n}(\mu)\leq W_{\infty,n}(\Gamma,1)$. Therefore \\$\liminf_{n\rightarrow\infty}W_{\infty,n}(\Gamma,1)\geq\sqrt{\nu(\mu_\Gamma)}$ and the lower bound $\sqrt{\nu(\mu_\Gamma)}$ is larger than 1 and it is rather explicit, see \eqref{expi}.

We combine Theorem \ref{mainn} and \eqref{genel} and state a result which can be considered a generalization of the Szeg\H{o} theorem on an interval as  $ \frac{[W_{2,n}(\mu)]^2}{S(\mu)}$ has the same limit for all $\mu\in \mathrm{Sz}(\Gamma)$.

\begin{corollary}\label{haydaa}
	Let $\mu\in \mathrm{Sz}(\Gamma)$. Then 
	\begin{align}
		1< \lim_{n\rightarrow \infty} [W_{2,n}(\mu_\Gamma)]^2= \lim_{n\rightarrow \infty} \frac{[W_{2,n}(\mu)]^2}{S(\mu)}\leq 2.
	\end{align}
The equality
\begin{align}\label{nu28}
\lim_{n\rightarrow \infty} \frac{[W_{2,n}(\mu)]^2}{S(\mu)}= 2
\end{align}
holds if and only if
\begin{align}\label{symt}
	g_{+}^{\prime}(z)=g_{-}^{\prime}(z)\,\, \mbox{ for all }  z\in \Gamma_0.
\end{align}

\end{corollary}
Widom proved in \cite[Theorem 11.4]{Wid69} that $\limsup_{n\rightarrow\infty}W_{\infty,n}(\Gamma,\rho)\leq 2 S(\rho)$ if $\rho$ is as in Theorem \ref{cheby}. He also conjectured in \cite[Section 11]{Wid69} that $\liminf_{n\rightarrow\infty}W_{\infty,n}(\Gamma,\rho)\geq 2 S(\rho)$. This conjecture was proven to be wrong as  $\lim_{n\rightarrow\infty}W_{\infty,n}(\Gamma, 1)<2$ if $\Gamma$ is a circular arc, see \cite{TotYud15,thi}. In the next theorem we show that  \eqref{upp2} holds under very mild conditions.

\begin{theorem}\label{cheby}
	Let $\rho$ be upper semicontinuous non-negative, bounded above on $\Gamma$ and $S(\rho)>0$. Then
	\begin{align}\label{upp}
		\limsup_{n\rightarrow \infty} W_{\infty,n}(\Gamma,\rho)\leq 2\sqrt{\pi R_{\mu_\Gamma}(\infty) \ca(\Gamma)} S(\rho).
	\end{align}
If there is a $z\in \Gamma_0$ such that 
\begin{align}\label{condi}
g_{+}^{\prime}(z)\neq g_{-}^{\prime}(z)
\end{align}

then 
	\begin{align}\label{upp2}
	\limsup_{n\rightarrow \infty} W_{\infty,n}(\Gamma,\rho)< 2   S(\rho).
\end{align}
\end{theorem}
The key idea in the proof of Theorem \ref{cheby} is to associate the extremal problems $W_{\infty,n}(\Gamma,\rho)$ and $W_{2,n}(\Gamma, \rho^2\,\mu_\Gamma)$. We also refer the reader to \cite{alper} for another application of this idea.

It is easy to see that if $\Gamma$ is an interval then \eqref{symt} holds and $W_{\infty,n}(\Gamma,1)=2$ for all $n$. The normal derivatives $g_{\pm}^{\prime}$ also play an important role in Bernstein and Markov inequalities, see \cite{act}. In the next theorem, we replace the condition \eqref{condi} with another one (non-analyticity of $\Gamma_0$) which is easier to check.
\begin{theorem}\label{easy}
	Let $\mu\in \mathrm{Sz}(\Gamma)$ and $\rho$ be upper semicontinuous non-negative, bounded above and $S(\rho)>0$. If $\Gamma_0$ is not analytic then
	\begin{enumerate}[$(i)$]
		\item $\lim_{n\rightarrow\infty}\frac{[W_{2,n}(\mu)]^2}{S(\mu)}< 2.$
		\item  	$\limsup_{n\rightarrow \infty} W_{\infty,n}(\Gamma,\rho)< 2   S(\rho).$
	\end{enumerate}
\end{theorem}
In the next section we prove Theorem \ref{mainn}, Theorem \ref{cheby} and Theorem \ref{easy}.

\section{Proofs}

\begin{proof}[Proof of Theorem \ref{mainn}]
	Without loss of generality, we assume that the endpoints $A, B$ of $\Gamma$ are $-1$ and $1$. Let 
	\begin{align}
		u=L(z)= z+\sqrt{z^2-1}
	\end{align}
	where $\sqrt{z^2-1}$ is asymptotically $z$ near $\infty$ and let
	\begin{align}
		z=\varphi(u)=\frac{1}{2}\left(u+\frac{1}{u}\right)
	\end{align}
	Then $L$ maps $\Gamma$ onto a $C^{2+}$ Jordan curve $\tilde{\Gamma}$, see  \cite[Section 11]{Wid69}. As $u$ traverses $\tilde{\Gamma}$ once, $z$ traverses the arc $\Gamma$ twice. Let $\Omega_{\tilde{\Gamma}}$ be the connected component of $\overline{\bb C}\setminus \tilde{\Gamma}$  which contains $\infty$.
	Then  $\overline{\bb C}\setminus \Gamma$ corresponds to $\Omega_{\tilde{\Gamma}}$ in the sense that $L$ and $\varphi$ are conformal maps from one of these domains onto the other and they both fix $\infty$.
	
	Let us follow the proof of \cite[theorem 12.3]{Wid69}. Since $\omega_\Gamma(z)\sqrt{|z-A||z-B|}$ is $C^{0+}$ and positive on $\Gamma$ it follows that $F_{\mu_\Gamma}(\varphi(\cdot))$ is analytic on $\Omega_{\tilde{\Gamma}}$ as a function of $u$ and it extends to a $C^{0+}$ function on $\tilde{\Gamma}$. Since $g_{\pm}^{\prime}(z)\sqrt{|z-A||z-B|}$  are also positive $C^{0+}$ functions it follows that there is a $C>0$ such that 
	\begin{align}
		\frac{1}{C}< \left\lvert\sqrt{\frac{\phi^{\prime}(\varphi(u))}{R_{\mu_\Gamma}(\varphi(u))}}\right\rvert<C\,\, \mbox{ for all } u\in \tilde{\Gamma}. 
	\end{align}
	If we apply maximum modulus principle for  $\sqrt{\frac{\phi^{\prime}(\varphi(\cdot))}{R_{\mu_\Gamma}(\varphi(\cdot))}}$ and  $\sqrt{\frac{R_{\mu_\Gamma}(\varphi(\cdot))}{\phi^{\prime}(\varphi(\cdot))}}$ we see that both of these functions are bounded in $\Omega_{\tilde{\Gamma}}$. Thus both $\sqrt{\frac{R_{\mu_\Gamma}}{\phi^{\prime}}}$ and $\sqrt{\frac{\phi^{\prime}}{R_{\mu_\Gamma}}}$ are  bounded in $\overline{\bb C}\setminus\Gamma$. 
	
	Let us first obtain the lower bound in \eqref{nu1}. We want to show that \\$2\pi R_{\mu_\Gamma}(\infty)/ \phi^{\prime}(\infty)>1$ or equivalently $\phi^{\prime}(\infty)/(2\pi R_{\mu_\Gamma}(\infty)) <1$. Since $g_{\pm}^{\prime}$ are \\bounded below away from $0$ and the expression in the bracket in \eqref{tirt3} is less than 1 for $u=\pm 1$ (the value at $1$ is defined as  the limit of the expression as $u\rightarrow 1$ with $u\in\tilde{\Gamma}\setminus\{-1,1\}$ and similarly for $u=-1$), the maximum modulus principle yields 
	\begin{align}\label{tirt3}
		\frac{\phi^{\prime}(\infty)}{2\pi R_{\mu_\Gamma}(\infty)}=\frac{\phi^{\prime}(\varphi(\infty))}{2\pi R_{\mu_\Gamma}(\varphi(\infty))}   \leq \sup_{u\in \tilde{\Gamma}}\left[\frac{g_{\pm}^{\prime}(\varphi(u))}{g_{+}^{\prime}(\varphi(u))+g_{-}^{\prime}(\varphi(u))}  \right] <1
	\end{align}
	which gives the lower bound part.
	
	The constant function $1$ is in $H_2(\overline{\bb C}\setminus \Gamma,\,\mu_\Gamma)$  since  $\sqrt{\frac{R_{\mu_\Gamma}(\Psi(\cdot))}{\phi^{\prime}(\Psi(\cdot))}}$  is bounded in $\overline{\bb C}\setminus \bb D$. For the trial function $F\equiv 1$ we have $\|F\|^2_{H_2(\overline{\bb C}\setminus \Gamma,\,\mu_\Gamma)}\leq 2$ since $\mu_\Gamma$ is a unit measure, see \eqref{normi}. Hence $\nu(\mu_\Gamma)\leq 2$.
	
	Assume that $\nu(\mu_\Gamma)=2$. Then the function $F\equiv 1$ is the solution of the extremal problem \eqref{pro1}. It follows from \eqref{extra} that
	\begin{align}
		\sqrt{\frac{\phi^{\prime}(z)}{R_{\mu_\Gamma}(z)}}\sqrt{\frac{R_{\mu_\Gamma}(\infty)}{\phi^{\prime}(\infty)}}=1
	\end{align}
	for all $z\in\overline{\bb C}\setminus \Gamma$ and also on the boundary. Since $|R_{\mu_\Gamma}|$ has identical boundary values on both sides of the arc, we also have $|\phi_{+}^{\prime}(z)|= |\phi_{-}^{\prime}(z)|$ ($ds$ a.e). Since $g_{\pm}^{\prime}$ are continuous on $\Gamma_0$ and \eqref{hadda} holds this implies \eqref{sym}.
	
		Conversely, let us assume \eqref{sym} holds. Then $\phi^{\prime}= \pi R_{\mu_\Gamma}$. This implies that $\nu(\mu_\Gamma)=2\pi R_{\mu_\Gamma}(\infty)/ \phi^{\prime}(\infty)=2$ which is \eqref{nu2}.
\end{proof}

\begin{proof}[Proof of Theorem \ref{cheby}]
	First let us assume that $\rho$ is $ C^{0+}$ and it is bounded below away from $0$. Let 
	\begin{align}
		d\mu= \rho^2  d\mu_\Gamma.
	\end{align}
	Here, $\rho^2(z) \omega_\Gamma(z)\sqrt{|z-A||z-B|}$ is positive and $C^{0+}$ on $\Gamma$. We follow the proof of  \cite[Theorem 12.3]{Wid69}. Let $C$ be a Jordan curve oriented counter-clockwise such that $z$ and $\Gamma$ are in the interior of $C$. Then
	\begin{align}
		Q_n(z):= \frac{1}{2\pi i }\int_C F_\mu(w) \phi^n(w) \frac{dw}{w-z}
	\end{align}
	is a polynomial of degree $n$ with leading coefficient $[\ca(\Gamma)]^{-n}$. It was shown in \cite[Lemma 11.2 and p. 221]{Wid69} that 
	\begin{align}
		\max_{z\in \Gamma}|Q_n(z)-F_{\mu,+}(z)\phi_{+}^n(z)- F_{\mu,-}(z)\phi_{-}^n(z)|\rightarrow 0
	\end{align}
	as $n\rightarrow\infty$. Note that $F_{\mu,+}$, $F_{\mu,-}$, $F_{\mu_\Gamma,+}$, $F_{\mu_\Gamma,-}$, $R_{\rho^2, +}$,  $R_{\rho^2, -}$ are all continuous on $\Gamma$.

	The function $|\sqrt{R_{\rho^2}}|$ has boundary values $\rho(z)$ on $\Gamma$ and  $\sqrt{R_{\rho^2}(\infty)}= R_{\rho}(\infty)=S(\rho)$. Let $\|\cdot \|_\Gamma$ denote the sup-norm on $\Gamma$. Since $W_{\infty,n}(\Gamma,\rho)\leq \|\rho\, Q_n\|_\Gamma$, using the inequality
	
	\begin{align}\label{tirt}
		\frac{\sqrt{a}+\sqrt{b}}{\sqrt{a+b}}\leq \sqrt{2}\, 	\mbox{ for all } a,b>0,
	\end{align}

	 in \eqref{ttttt}, we get
	
	\begin{align}
		\limsup_{n\rightarrow\infty} W_{\infty,n}(\Gamma,\rho)&\leq \limsup_{n\rightarrow\infty} \left[\sup _{z\in\Gamma}|\rho(z)(F_{\mu,+}(z)\phi_{+}^n(z)+F_{\mu,-}(z)\phi_{-}^n(z)|     \right]\\
		&\leq \sup _{z\in\Gamma}\left[\rho(z)(|F_{\mu,+}(z)|+|F_{\mu,-}(z)|   )\right]\\
		&= \sqrt{\frac{R_{\mu}(\infty)}{\phi^{\prime}(\infty)}}   \sup _{z\in\Gamma_0}\left[\rho(z) \left[\left\lvert \sqrt{\frac{\phi_{+}^{\prime}(z)}{R_{\mu}(z)}}  \right\rvert  + \left\lvert \sqrt{\frac{\phi_{-}^{\prime}(z)}{R_{\mu}(z)}}  \right\rvert    \right]\right]\\
		&=\sqrt{ R_{\rho^2}(\infty)\, R_{\mu_\Gamma}(\infty)\, \ca(\Gamma) } \sup _{z\in\Gamma_0}\left[\rho(z) \frac{|\sqrt{\phi_{+}^\prime(z)}|+ |\sqrt{\phi_{+}^\prime(z)}|}{|\sqrt{R_{\rho^2}(z)\, R_{\mu_\Gamma(z)} }|}
		\right]\\
		&=\sqrt{2\pi R_{\mu_\Gamma}(\infty) \ca(\Gamma)} S(\rho)  \sup _{z\in\Gamma_0} \frac{\sqrt{g^\prime_{+}(z)}+ \sqrt{g^\prime_{-}(z)}}{\sqrt{ g^\prime_{+}(z)+g^\prime_{-}(z)}}\label{sahi}\\
		&\leq 2\sqrt{\pi R_{\mu_\Gamma}(\infty) \ca(\Gamma)} S(\rho).\label{ttttt}
	\end{align}
	Thus we obtain \eqref{upp} if $\rho$ is $C^{0+}$ and away from $0$. 
	
	Let us remove these restrictions. Let us assume that $\rho$ is non-negative and $S(\rho)>0$. Since $\rho$ is upper semicontinuous there is a sequence $\rho_m$ of $C^{0+}$ positive functions on $\Gamma$ such that $\rho_m\downarrow \rho$. By Lebesgue's monotone convergence theorem $\int \log{\rho_m}\,d\mu_\Gamma \downarrow \int \log{\rho}\,d\mu_\Gamma$ and thus 
	\begin{align}\label{slii}
		S(\rho_m)\downarrow S(\rho)
	\end{align}
	as  $m\rightarrow\infty$. Since $W_{\infty,n}(\Gamma,\rho)\leq W_{\infty,n}(\Gamma,\rho_m)$ for all $m,n$, \eqref{ttttt} yields
	\begin{align}
		\limsup_{n\rightarrow\infty} W_{\infty,n}(\Gamma,\rho) \leq 2\sqrt{\pi R_{\mu_\Gamma}(\infty) \ca(\Gamma)} S(\rho_m)
	\end{align}
	and since \eqref{slii} holds we get \eqref{upp}.
	
	If there is a $z\in \Gamma_0$ such that  $g_{+}^{\prime}(z)\neq g_{-}^{\prime}(z)$ then $\pi R_{\mu_\Gamma}(\infty)\ca(\Gamma)<1$ by Theorem \ref{mainn}.  Thus  \eqref{upp2} follows from \eqref{upp} in this case.
\end{proof}

\begin{proof}Proof of Theorem \ref{easy}
	We show that $C^{2+}$ smoothness of $\Gamma$ and the condition \eqref{symt} imply analyticity of $\Gamma_0$.  This is indeed a known result. We summarize the ideas in \cite[Proposition 4]{stahl} and \cite[p. 87-88]{yatt}.
	
Assume that \eqref{symt} is satisfied.	Let $U$ be a simply connected domain such that $\Gamma_0\subset U$ and $U\setminus \Gamma_0= U_1 \cup U_2$  where $U_1$ and $U_2$ are two disjoint subdomains of $U$. Then the function
	\[ L(z)=\begin{cases} 
		g_{\Gamma}(z) & z\in U_1 \cup \Gamma_0 \\
		-g_{\Gamma}(z)& z\in U_2
	\end{cases}
	\]
	is harmonic on $U$ (clearly the condition \eqref{symt} is necessary for such an extension of the Green function). Let $\tilde{L}$ be a harmonic conjugate of $L$. Then $H:= L+i\tilde{L}$ is analytic on $U$ and $H^{\prime}=L_x-iL_y$ is non-zero on $\Gamma_0$, see \cite[p. 88]{yatt}. Since $L(z)=0$ on $\Gamma_0$ it follows that $H$ is injective on $\Gamma_0$ and $iH(\Gamma_0)=(c,d)$ where $c,d\in \bb R$, $c\neq d$. This implies that $\Gamma_0$ is analytic which is a contradiction. 
	
	Hence there is a $z\in \Gamma_0$ such that $g_{+}^{\prime}(z)\neq g_{-}^{\prime}(z)$.
	Thus parts $(i)$ and $(ii)$ follow from Corollary \ref{haydaa} and Theorem \ref{cheby}.
\end{proof}



\end{document}